\documentclass[12pt]{article}
\usepackage{amsmath,amsfonts,amssymb,amsthm}

\theoremstyle{plain}
\newtheorem{theorem}{Theorem}[section]
\newtheorem{corollary}[theorem]{Corollary}
\newtheorem{lemma}[theorem]{Lemma}
\newtheorem{proposition}[theorem]{Proposition}

\theoremstyle{definition}
\newtheorem{definition}[theorem]{Definition}

\newtheorem{example}[theorem]{Example}
\newtheorem{remark}[theorem]{Remark}

\newtheorem*{openproblem}{Open Problem}

\numberwithin{equation}{section}

\newcommand{\R}{{\mathbb R}}
\newcommand{\N}{{\mathbb N}}

\newcommand{\Om}{\Omega}

\providecommand{\vint}[1]{\mathchoice
          {\mathop{\vrule width 5pt height 3 pt depth -2.5pt
                  \kern -9pt \kern 1pt\intop}\nolimits_{\kern -5pt{#1}}}
          {\mathop{\vrule width 5pt height 3 pt depth -2.6pt
                  \kern -6pt \intop}\nolimits_{\kern -3pt{#1}}}
          {\mathop{\vrule width 5pt height 3 pt depth -2.6pt
                  \kern -6pt \intop}\nolimits_{\kern -3pt{#1}}}
          {\mathop{\vrule width 5pt height 3 pt depth -2.6pt
                  \kern -6pt \intop}\nolimits_{\kern -3pt{#1}}}}

\newcommand{\eps}{\varepsilon}
\newcommand{\loc}{\mathrm{loc}}

\newcommand{\BV}{\mathrm{BV}}
\newcommand{\liploc}{\mathrm{Lip}_{\mathrm{loc}}}

\newcommand{\ch}{\text{\raise 1.3pt \hbox{$\chi$}\kern-0.2pt}}

\DeclareMathOperator{\Mod}{Mod}
\DeclareMathOperator{\capa}{Cap}
\DeclareMathOperator{\rcapa}{cap}

\DeclareMathOperator{\diam}{diam}

\DeclareMathOperator{\Lip}{Lip}

\DeclareMathOperator{\inte}{int}
\DeclareMathOperator{\fint}{fine-int}
\DeclareMathOperator{\0fint}{0-fine-int}

\begin{document}
\title{A Federer-style characterization of sets of \\
finite perimeter  on metric spaces
\footnote{
\hfill \break {\it Keywords\,}: metric measure space, bounded variation, finite perimeter, Federer's characterization, measure theoretic boundary, fine boundary.
}}
\author{Panu Lahti}
\maketitle

\begin{abstract}
In the setting of a metric space equipped with a doubling measure that supports a Poincar\'e inequality, we show that a set $E$ is of finite perimeter if and only if $\mathcal H(\partial^1 I_E)<\infty$, that is, if and only if the
codimension one Hausdorff measure of the \emph{$1$-fine boundary} of the set's measure theoretic interior $I_E$ is finite.
\end{abstract}

\section{Introduction}

Federer's structure theorem states that a set $E\subset\R^k$ is of finite perimeter if and only if $\mathcal{H}(\partial^*E)$
is finite, see \cite[Section 4.5.11]{Fed}. Here $\mathcal H$ is the codimension one (in this case, $k-1$-dimensional) Hausdorff measure, and $\partial^*E$ is the measure theoretic boundary of $E$.
In a complete metric space $X$
equipped with a doubling measure that supports a Poincar\'e inequality, the ``only if'' direction has been shown by Ambrosio,
see \cite{A1}, but
the ``if'' direction remains open.

In this paper we define for $A\subset X$ the \emph{1-fine boundary} $\partial^1 A$, which always contains $\partial^*A$ but can be strictly larger; for example on the real line, the $1$-fine boundary coincides with the topological boundary. However, using a fine continuity result for $\BV$ functions given in \cite{L}, we show that for any set of finite perimeter $E$, denoting the measure theoretic interior of $E$ by $I_E$, the difference $\partial^1 I_E\setminus \partial^*E$ is $\mathcal H$-negligible. In particular, then $\mathcal{H}(\partial^1 I_E)<\infty$. In showing this, we first prove a suitable characterization of the $1$-fine boundary, in analogy with what is known in the case $p>1$, see \cite[Section 7]{BB-OD}.

Then we show that the condition $\mathcal H(\partial^1 I_E)<\infty$ is also sufficient for $E$ to be of finite perimeter. For this, we generalize further concepts and results of fine potential theory from the case $p>1$ to the case $p=1$; all such considerations appear to be new even in the Euclidean setting. In particular, we study the existence of \emph{capacitary potentials} and prove weak analogs of the \emph{Cartan property} for solutions of obstacle problems, and of the \emph{Choquet property} for finely open sets. These have recently been studied for $p>1$ in the metric setting in \cite{BBL-CCK,BBL-WC}; see also \cite{MZ} and \cite{HKM} for the Euclidean theory and its history in the unweighted and weighted settings, respectively.

Our result  is the following --- see Section \ref{sec:prelis} for the definitions.

\begin{theorem}
For an open set $\Omega\subset X$ and a  $\mu$-measurable set $E\subset X$, we have
$P(E,\Omega)<\infty$ if and only if $\mathcal H(\partial^1 I_E\cap \Omega)<\infty$.
Furthermore, then
$\mathcal H((\partial^1 I_E\setminus \partial^*E)\cap\Omega)=0$.
\end{theorem}

Necessity is given by Theorem \ref{thm:finite perimeter implies finite boundary} in Section \ref{sec:necessity}. Sufficiency is given by Theorem \ref{thm:sufficiency} in Section \ref{sec:sufficiency}. The results of \cite{LaSh} and \cite{L} are used extensively in the proofs.

\paragraph{Acknowledgments.} The research was
funded by a grant from the Finnish Cultural Foundation.
The author wishes to thank Nageswari Shan\-mu\-ga\-lingam for reading the manuscript and
providing useful comments.

\section{Preliminaries}\label{sec:prelis}

In this section we introduce the notation, definitions, and assumptions used in the paper.

In this paper, $(X,d,\mu)$ is a complete metric space equipped
with a Borel regular outer measure $\mu$ satisfying a doubling property, that is,
there is a constant $C_d\ge 1$ such that
\[
0<\mu(B(x,2r))\leq C_d\,\mu(B(x,r))<\infty
\]
for every ball $B=B(x,r)$ with center $x\in X$ and radius $r>0$. We assume that $X$ consists of at least two points.
By iterating the doubling condition, we obtain that for any $x\in X$ and $y\in B(x,R)$ with $0<r\le R<\infty$, we have
\begin{equation}\label{eq:homogenous dimension}
\frac{\mu(B(y,r))}{\mu(B(x,R))}\ge \frac{1}{C_d^2}\left(\frac{r}{R}\right)^{Q},
\end{equation}
where $Q>1$ only depends on the doubling constant $C_d$.
When we want to specify that a constant $C$
depends on the parameters $a,b, \ldots,$ we write $C=C(a,b,\ldots)$.

A complete metric space with a doubling measure is proper,
that is, closed and bounded sets are compact. Since $X$ is proper, for any open set $\Omega\subset X$
we define $\liploc(\Omega)$ to be the space of
functions that are Lipschitz in every open $\Omega'\Subset\Omega$.
Here $\Omega'\Subset\Omega$ means that $\overline{\Omega'}$ is a
compact subset of $\Omega$. Other local spaces of functions are defined similarly.

For any set $A\subset X$ and $0<R<\infty$, the restricted spherical Hausdorff content
of codimension one is defined by
\[
\mathcal{H}_{R}(A):=\inf\left\{ \sum_{i=1}^{\infty}
  \frac{\mu(B(x_{i},r_{i}))}{r_{i}}:\,A\subset\bigcup_{i=1}^{\infty}B(x_{i},r_{i}),\,r_{i}\le R\right\}.
\]
The codimension one Hausdorff measure of $A\subset X$ is given by
\begin{equation*}
  \mathcal{H}(A):=\lim_{R\rightarrow 0}\mathcal{H}_{R}(A).
\end{equation*}

The measure theoretic boundary $\partial^{*}E$ of a set $E\subset X$ is the set of points $x\in X$
at which both $E$ and its complement have positive upper density, i.e.
\[
\limsup_{r\to 0}\frac{\mu(B(x,r)\cap E)}{\mu(B(x,r))}>0\quad\;
  \textrm{and}\quad\;\limsup_{r\to 0}\frac{\mu(B(x,r)\setminus E)}{\mu(B(x,r))}>0.
\]
The measure theoretic interior and exterior of $E$ are defined respectively by
\begin{equation}\label{eq:definition of measure theoretic interior}
I_E:=\left\{x\in X:\,\lim_{r\to 0}\frac{\mu(B(x,r)\setminus E)}{\mu(B(x,r))}=0\right\}
\end{equation}
and
\begin{equation}\label{eq:definition of measure theoretic exterior}
O_E:=\left\{x\in X:\,\lim_{r\to 0}\frac{\mu(B(x,r)\cap E)}{\mu(B(x,r))}=0\right\}.
\end{equation}
Note that we always have a partitioning of the space into the disjoint sets
$\partial^*E$, $I_E$, and $O_E$. Moreover, $I_{X\setminus E}=O_E$.
By the Lebesgue differentiation theorem, for a \emph{$\mu$-measurable} set $E$ we have
$\mu(E\Delta I_E)=0$ and $\mu((X\setminus E)\Delta O_E)=0$, where $\Delta$ is the symmetric difference, and so $I_E=I_{I_E}$ and $\partial^*E=\partial^* I_E$.

A curve is a rectifiable continuous mapping from a compact interval
into $X$.
The length of a curve $\gamma$
is denoted by $\ell_{\gamma}$. We will assume every curve to be parametrized
by arc-length, which can always be done (see e.g. \cite[Theorem~3.2]{Hj}).
A nonnegative Borel function $g$ on $X$ is an upper gradient 
of an extended real-valued function $u$
on $X$ if for all curves $\gamma$, we have
\begin{equation}\label{eq:definition of upper gradient}
|u(x)-u(y)|\le \int_\gamma g\,ds,
\end{equation}
where $x$ and $y$ are the end points of $\gamma$. We interpret $|u(x)-u(y)|=\infty$ whenever  
at least one of $|u(x)|$, $|u(y)|$ is infinite.
Of course, by replacing $X$ with a set $A\subset X$ and considering curves $\gamma$ in $A$, we can talk about a function $g$ being an upper gradient of $u$ in $A$.
Upper gradients were originally introduced in~\cite{HK}.

If $g$ is a nonnegative $\mu$-measurable function on $X$
and (\ref{eq:definition of upper gradient}) holds for $1$-almost every curve,
we say that $g$ is a $1$-weak upper gradient of~$u$. 
A property holds for $1$-almost every curve
if it fails only for a curve family with zero $1$-modulus. 
A family $\Gamma$ of curves is of zero $1$-modulus if there is a 
nonnegative Borel function $\rho\in L^1(X)$ such that 
for all curves $\gamma\in\Gamma$, the curve integral $\int_\gamma \rho\,ds$ is infinite.

Given an open set $\Omega\subset X$, we consider the following norm
\[
\Vert u\Vert_{N^{1,1}(\Omega)}:=\Vert u\Vert_{L^1(\Omega)}+\inf \Vert g\Vert_{L^1(\Omega)},
\]
where the infimum is taken over all $1$-weak upper gradients $g$ of $u$ in $\Omega$.
The substitute for the Sobolev space $W^{1,1}(\Omega)$ in the metric setting is the Newton-Sobolev space
\[
N^{1,1}(\Omega):=\{u:\|u\|_{N^{1,1}(\Omega)}<\infty\}.
\]
We understand Newton-Sobolev functions to be defined everywhere (even though $\Vert \cdot\Vert_{N^{1,1}(\Omega)}$ is, precisely speaking, then only a seminorm).
For more on Newton-Sobolev spaces, we refer to~\cite{S, BB, HKST}.

Next we recall the definition and basic properties of functions
of bounded variation on metric spaces, following \cite{M}. See also e.g. \cite{AFP, EvaG92, Giu84, Zie89} for the classical 
theory in the Euclidean setting.
For $u\in L^1_{\loc}(X)$, we define the \emph{total variation} of $u$ in $X$ by
\[
\|Du\|(X):=\inf\left\{\liminf_{i\to\infty}\int_X g_{u_i}\,d\mu:\, u_i\in \Lip_{\loc}(X),\, u_i\to u\textrm{ in } L^1_{\loc}(X)\right\},
\]
where each $g_{u_i}$ is an upper gradient of $u_i$.
We say that a function $u\in L^1(X)$ is \emph{of bounded variation}, 
and denote $u\in\BV(X)$, if $\|Du\|(X)<\infty$.
By replacing $X$ with an open set $\Omega\subset X$ in the definition of the total variation, we can define $\|Du\|(\Omega)$.
For an arbitrary set $A\subset X$, we define
\[
\|Du\|(A)=\inf\{\|Du\|(\Omega):\, A\subset\Omega,\,\Omega\subset X
\text{ is open}\}.
\]
If $\Vert Du\Vert(X)<\infty$, $\|Du\|(\cdot)$ is a finite Radon measure on $X$ by \cite[Theorem 3.4]{M}.
A $\mu$-measurable set $E\subset X$ is said to be of \emph{finite perimeter} if $\|D\ch_E\|(X)<\infty$, where $\ch_E$ is the characteristic function of $E$.
The perimeter of $E$ in $\Omega$ is also denoted by
\[
P(E,\Omega):=\|D\ch_E\|(\Omega).
\]
Similarly as above, if $P(E,\Omega)<\infty$, then $P(E,\cdot)$ is finite Radon measure
on $\Omega$.

For any Borel sets $E_1,E_2\subset X$, we have by \cite[Proposition 4.7]{M}
\[
P(E_1\cup E_2,X)\le P(E_1,X)+P(E_2,X). 
\]
The proof works equally well for $\mu$-measurable $E_1,E_2\subset X$ and with $X$
replaced by any open set $\Omega$. Then by approximation from the outside by open sets, we obtain for any $A\subset X$
\begin{equation}\label{eq:sets of finite perimeter form an algebra}
P(E_1\cup E_2,A)\le P(E_1,A)+P(E_2,A). 
\end{equation}

We have the following coarea formula from~\cite[Proposition 4.2]{M}: if $U\subset X$ is an open set and $u\in L^1_{\loc}(U)$, then
\begin{equation}\label{eq:coarea}
\|Du\|(U)=\int_{-\infty}^{\infty}P(\{u>t\},U)\,dt.
\end{equation}

We will assume throughout that $X$ supports a $(1,1)$-Poincar\'e inequality,
meaning that there exist constants $C_P\ge 1$ and $\lambda \ge 1$ such that for every
ball $B(x,r)$, every locally integrable function $u$ on $X$,
and every upper gradient $g$ of $u$,
we have 
\[
\vint{B(x,r)}|u-u_{B(x,r)}|\, d\mu 
\le C_P r\vint{B(x,\lambda r)}g\,d\mu,
\]
where 
\[
u_{B(x,r)}:=\vint{B(x,r)}u\,d\mu :=\frac 1{\mu(B(x,r))}\int_{B(x,r)}u\,d\mu.
\]
The $(1,1)$-Poincar\'e inequality implies the so-called Sobolev-Poincar\'e inequality, see e.g. \cite[Theorem 4.21]{BB}, and by applying the latter to approximating locally Lipschitz functions in the definition of the total variation, we get the following Sobolev-Poincar\'e inequality for $\BV$ functions. For every ball $B(x,r)$ and every $u\in L^1_{\loc}(X)$, we have
\[
\left(\,\vint{B(x,r)}|u-u_{B(x,r)}|^{Q/(Q-1)}\,d\mu\right)^{(Q-1)/Q}
\le C_{SP}r\frac{\Vert Du\Vert (B(x,2\lambda r))}{\mu(B(x,2\lambda r))},
\]
where $Q$ is the exponent from \eqref{eq:homogenous dimension} and
 $C_{SP}=C_{SP}(C_d,C_P,\lambda)\ge 1$ is a constant.
For a $\mu$-measurable set $E\subset X$, this implies (see e.g. \cite[Equation (3.1)]{KoLa})
\[
\frac 12 \left(\frac{\min\{\mu(B(x,r)\cap E),\mu(B(x,r)\setminus E)\}}{\mu(B(x,r))}\right)^{(Q-1)/Q}
\le C_{SP}r\frac{P(E,B(x,2\lambda r))}{\mu(B(x,2\lambda r))}.
\]
Rearranged, this implies
\begin{equation}\label{eq:relative isoperimetric inequality}
\begin{split}
&\min\{\mu(B(x,r)\cap E),\mu(B(x,r)\setminus E)\}\\
&\quad \le 2 C_{SP}r\left(\frac{\min\{\mu(B(x,r)\cap E),\mu(B(x,r)\setminus E)\}}{\mu(B(x,r))}\right)^{1/Q}P(E,B(x,2\lambda r)).
\end{split}
\end{equation}
Moreover, the $(1,1)$-Poincar\'e inequality implies the following Sobolev inequality.
If $x\in X$, $0<r<\frac{1}{4}\diam(X)$, and $u\in N^{1,1}(X)$ with $u=0$ in $X\setminus B(x,r)$, then
\begin{equation}\label{eq:sobolev inequality}
\int_{B(x,r)} |u|\,d\mu\le C_S r \int_{B(x,r)}  g_u\,d\mu
\end{equation}
for any upper gradient $g_u$ of $u$ and a constant $C_S=C_S(C_d,C_P)\ge 1$, see \cite[Theorem 5.51]{BB}.
By approximation, we obtain that for any $x\in X$, any $0<r<\frac{1}{4}\diam(X)$, and any $\mu$-measurable set $E\subset B(x,r)$, we have
\begin{equation}\label{eq:isop inequality with zero boundary values}
\mu(E)\le C_S r P(E,X).
\end{equation}

The $1$-capacity of a set $A\subset X$ is given by
\begin{equation}\label{eq:definition of p-capacity}
 \capa_1(A):=\inf \Vert u\Vert_{N^{1,1}(X)},
\end{equation}
where the infimum is taken over all functions $u\in N^{1,1}(X)$ such that $u\ge 1$ in $A$.
We know that $\capa_1$ is an outer capacity, meaning that
\[
\capa_1(A)=\inf\{\capa_1(U):\,U\supset A\textrm{ is open}\}
\]
for any $A\subset X$, see e.g. \cite[Theorem 5.31]{BB}. If a property holds outside a set
$A\subset X$ with $\capa_1(A)=0$, we say that it holds $1$-quasieverywhere.
If $u\in N^{1,1}(X)$, then $\Vert u-v\Vert_{N^{1,1}(X)}=0$ if and only if $u=v$ $1$-quasieverywhere, see \cite[Proposition 1.61]{BB}.

The variational $1$-capacity of a set $A\subset D$ with respect to a set $D\subset X$ is given by
\[
\rcapa_1(A,D):=\inf \int_{X} g_u\,d\mu,
\]
where the infimum is taken over functions $u\in N^{1,1}(X)$ and upper gradients $g_u$ of $u$ such that $u\ge 1$ in $A$ (equivalently,  $1$-quasieverywhere in $A$) and $u=0$ in $X\setminus D$.
We know that $\rcapa_1$ is also an outer capacity, in the sense that if $\Omega\subset X$ is a bounded open set and $A\Subset \Omega$, then
\[
\rcapa_1(A,\Omega)=\inf\{\rcapa_1(U):\,U\textrm{ open},\,A\subset U\subset\Omega \},
\]
see \cite[Theorem 6.19]{BB}.
For basic properties satisfied by capacities, such as monotonicity and countable subadditivity, see e.g. \cite{BB}.

Given a set $E\subset X$ of finite perimeter, for $\mathcal H$-almost every $x\in \partial^*E$ we have
\begin{equation}\label{eq:definition of gamma}
\gamma \le \liminf_{r\to 0} \frac{\mu(B(x,r)\cap E)}{\mu(B(x,r))} \le \limsup_{r\to 0} \frac{\mu(B(x,r)\setminus E)}{\mu(B(x,r))}\le 1-\gamma,
\end{equation}
where $\gamma \in (0,1/2]$ only depends on the doubling constant and the constants in the Poincar\'e inequality, 
see~\cite[Theorem 5.4]{A1}.
For an open set $\Omega\subset X$ and a $\mu$-measurable set $E\subset X$ with $P(E,\Omega)<\infty$, we know that for any Borel set $A\subset\Omega$,
\begin{equation}\label{eq:def of theta}
P(E,A)=\int_{\partial^{*}E\cap A}\theta_E\,d\mathcal H,
\end{equation}
where
$\theta_E\colon X\to [\alpha,C_d]$ with $\alpha=\alpha(C_d,C_P,\lambda)>0$, see \cite[Theorem 5.3]{A1} 
and \cite[Theorem 4.6]{AMP}.

The lower and upper approximate limits of an extended real-valued function $u$ on $X$ are defined respectively by
\begin{equation}\label{eq:lower approximate limit}
u^{\wedge}(x):
=\sup\left\{t\in\overline\R:\,\lim_{r\to 0}\frac{\mu(B(x,r)\cap\{u<t\})}{\mu(B(x,r))}=0\right\}
\end{equation}
and
\begin{equation}\label{eq:upper approximate limit}
u^{\vee}(x):
=\inf\left\{t\in\overline\R:\,\lim_{r\to 0}\frac{\mu(B(x,r)\cap\{u>t\})}{\mu(B(x,r))}=0\right\}.
\end{equation}

Note that for $u=\ch_E$ with $E\subset X$, we have $x\in I_E$ if and only if $u^{\wedge}(x)=u^{\vee}(x)=1$, $x\in O_E$ if and only if $u^{\wedge}(x)=u^{\vee}(x)=0$, and $x\in \partial^*E$ if and only if $u^{\wedge}(x)=0$ and $u^{\vee}(x)=1$.

We understand $\BV$ functions to be $\mu$-equivalence classes. To consider fine properties, we need to consider the pointwise representatives $u^{\wedge}$ and $u^{\vee}$.

We need to generalize concepts of fine potential theory from the case $p>1$ to the case $p=1$.
The following definition can be taken directly from e.g. \cite{BBL-WC}.
\begin{definition}
A set $A\subset X$ is \emph{$1$-quasiopen} if for every $\eps>0$ there is an
open set $G\subset X$ with $\capa_1(G)<\eps$ such that $A\cup G$ is open.
\end{definition}

Next we define the fine topology in the case $p=1$.
\begin{definition}\label{def:1 fine topology}
We say that $A\subset X$ is $1$-thin at the point $x\in X$ if
\[
\lim_{r\to 0}r\frac{\rcapa_1(B(x,r)\cap A,B(x,2r))}{\mu(B(x,r))}=0.
\]
We also say that a set $U\subset X$ is $1$-finely open if $X\setminus U$ is $1$-thin at every $x\in U$. Then we define the $1$-fine topology as the collection of $1$-finely open sets on $X$.

We denote the $1$-fine interior of a set $G\subset X$, i.e. the largest $1$-finely open set contained in $G$, by $\fint G$. We denote the $1$-fine closure of a set $G\subset X$, i.e. the smallest $1$-finely closed set containing $G$, by $\overline{G}^1$. We define the $1$-fine boundary of a set $G\subset X$ by
$\partial^1 G:=\overline{G}^1\setminus \fint G$.

Finally, we define the \emph{$1$-base} $b_1 G$ of a set $G\subset X$ as the set of points where $G$ is \emph{not} $1$-thin.
\end{definition}

Note that always $b_1 G\subset \overline{G}^1$. See \cite[Section 4]{L} for motivation of the definition of $1$-thinness, and for a proof of the fact that the $1$-fine topology is indeed a topology.

\section{The $1$-fine boundary}\label{sec:fine boundary}

In this section we give a suitable characterization of the $1$-fine boundary.

\begin{lemma}\label{lem:inclusion for fine boundaries}
For any $A\subset X$, we have $I_A\cup \partial^*A\subset b_1 A$ and
$\partial^*A\subset \partial^{1} A$.
\end{lemma}

\begin{proof}
Suppose $x\in X\setminus b_1 A$, so that
\[
\lim_{r\to 0} r\frac{\rcapa_1(B(x,r)\cap A, B(x,2r))}{\mu(B(x,r))}=0.
\]
By the definition of the variational capacity, for every $r>0$ we find a nonnegative function $u_r\in N^{1,1}(X)$ such that $u_r\ge 1$ in $B(x,r)\cap A$, $u_r=0$ in $X\setminus B(x,2r)$, and
$u_r$ has an upper gradient $g_r$ with
\[
\frac{r}{\mu(B(x,r))}\int_X g_r\,d\mu\to 0\qquad\textrm{as }r\to 0.
\]
But by the Sobolev inequality \eqref{eq:sobolev inequality},
\begin{align*}
\limsup_{r\to 0}\frac{\mu(B(x,r)\cap A)}{\mu(B(x,r))}
&\le \limsup_{r\to 0}C_d\vint{B(x,2r)}u_r\,d\mu\\
&\le \limsup_{r\to 0}2C_S C_d r\vint{B(x,2 r)}g_r\,d\mu=0.
\end{align*}
Thus $x\notin I_A\cup \partial^*A$.

To prove the second claim, note that by the first claim,
\[
\overline{A}^1\supset b_1 A \supset I_A\cup \partial^*A
\]
and
\begin{align*}
\fint A
&=X\setminus \overline{X\setminus A}^1\subset X\setminus b_1(X\setminus A)\\
&\subset X\setminus (I_{X\setminus A}\cup \partial^*A)
 = X\setminus (O_{A}\cup \partial^*A)=I_A.
\end{align*}
By combining these, we obtain
$\partial^1 A= \overline{A}^1\setminus \fint A\supset \partial^*A$.
\end{proof}

Next we gather some known results.

\begin{lemma}[{\cite[Lemma 4.3]{L}}]\label{lem:local boxing inequality}
Let $x\in X$, let $r>0$, and let $G\subset X$ be a $\mu$-measurable set with
\begin{equation}\label{eq:little G}
\frac{\mu(B(x,2r)\cap G)}{\mu(B(x,2r))}\le \frac{1}{2C_d^{\lceil\log_2 ( 128\lambda)\rceil}}.
\end{equation}
Then for some constant $C_1=C_1(C_d,C_P,\lambda)$, 
\[
\rcapa_1 (B(x,r)\cap I_G,B(x,2r))\le C_1 P(G,B(x,2 r)).
\]
\end{lemma}

Moreover, it is straightforward to show that for any set $A\subset X$ and any ball $B(x,r)$,
\begin{equation}\label{eq:capacity and Hausdorff measure}
\rcapa_1(B(x,r)\cap A,B(x,2r))\le C_d \mathcal H(B(x,r)\cap´A).
\end{equation}
This can be deduced by using suitable cutoff functions.

In the particular case of a set of finite perimeter, the $1$-fine closure is essentially just the measure theoretic closure; this is essentially contained in
\cite[Proposition 4.4]{L}, but we repeat the proof here.

\begin{lemma}\label{lem:fine closure of a set of finite perimeter}
Let $E\subset X$ be a set of finite perimeter. Then there exists a
$\mathcal H$-negligible set $N\subset X$ such that
\[
\overline{I_E}^1\subset I_E\cup \partial^*E \cup N.
\]
\end{lemma}

\begin{proof}
By \cite[Theorem 2.4.3]{AT} we know that if $\nu$ is a Radon measure on $X$, $t>0$, and $A\subset X$ is a Borel set for which we have
\[
\limsup_{r\to 0}r\frac{\nu(B(x,r))}{\mu(B(x,r))}\ge t
\]
for all $x\in A$, then $\nu(A)\ge t\mathcal H(A)$.
Since $E$ is of finite perimeter, we have $\mathcal H(\partial^*E)<\infty$ by \eqref{eq:def of theta}. By using \eqref{eq:capacity and Hausdorff measure} and the above density result with $\nu=\mathcal H|_{\partial^*E}$, we get
\begin{equation}\label{eq:estimate for density of measure theoretic boundary of E}
\limsup_{r\to 0}r\frac{\rcapa_1(B(x,r) \cap\partial^*E ,B(x,2r))}{\mu(B(x,r))}
\le C_d\limsup_{r\to 0}r\frac{\mathcal H(B(x,r)\cap\partial^*E )}{\mu(B(x,r))}=0
\end{equation}
for $\mathcal H$-almost every $x\in X\setminus \partial^* E$, that is,
for every $x\in X\setminus (\partial^*E\cup N)$ with $\mathcal H(N)=0$.

By Lemma \ref{lem:local boxing inequality}, if $x\in X$ and $r>0$ satisfy
\[
\frac{\mu(B(x,2r)\cap E )}{\mu(B(x,2r))}\le \frac{1}{2C_d^{\lceil\log_2 ( 128\lambda)\rceil}},
\]
then $\rcapa_1 (B(x,r)\cap I_E ,B(x,2r))\le C_1 P(E,B(x,2 r))$.
Thus we get for all $x\in X\setminus (I_E\cup \partial^*E\cup N)$
\begin{align*}
\limsup_{r\to 0}r\frac{\rcapa_1 (B(x,r) \cap I_E ,B(x,2r))}{\mu(B(x,r))}
&\le C_1\limsup_{r\to 0}r\frac{P(E,B(x,2 r))}{\mu(B(x,r))}\\
&\overset{\eqref{eq:def of theta}}{\le} C_1 C_d\limsup_{r\to 0}r\frac{\mathcal H(\partial^*E\cap B(x,2 r))}{\mu(B(x,r))}\\
&\le C_1 C_d^2\limsup_{r\to 0}r\frac{\mathcal H(\partial^*E\cap B(x,2 r))}{\mu(B(x,2r))}\\
&= 0
\end{align*}
by the equality in \eqref{eq:estimate for density of measure theoretic boundary of E}.
By combining this with \eqref{eq:estimate for density of measure theoretic boundary of E}, and \eqref{eq:capacity and Hausdorff measure} with $A=N$, we have
\[
\limsup_{r\to 0}r\frac{\rcapa_1 (B(x,r)\cap(I_E\cup \partial^*E\cup N), B(x,2r))}{\mu(B(x,r))}
=0
\]
for all $x\in X\setminus (I_E\cup \partial^*E\cup N)$.
Thus $I_E\cup \partial^*E\cup N\supset I_E$ is a $1$-finely closed set, so that $\overline{I_E}^1\subset  I_E\cup \partial^*E\cup N$.
\end{proof}

\begin{theorem}[{\cite[Corollary 5.4]{L}}]\label{thm:fine continuity of Newtonian functions}
Let $u\in N^{1,1}(X)$. Then $u$ is continuous (in the sense of a real-valued function) with respect to the $1$-fine topology $1$-quasieverywhere.
\end{theorem}

By \cite[Theorem 4.3, Theorem 5.1]{HaKi} we know that if $A\subset X$,
\begin{equation}\label{eq:null sets of Hausdorff measure and capacity}
\capa_1(A)=0\quad\ \textrm{if and only if}\quad\ \mathcal H(A)=0.
\end{equation}

\begin{lemma}\label{lem:capacity of fine closure preliminary result}
Let $x\in X$, let $0<r<\diam(X)/8$, and let $A\subset B(x,r)$ with
\begin{equation}\label{eq:a priori smallness of capacity}
r\frac{\rcapa_1(A,B(x,2r))}{\mu(B(x,r))}\le \frac{1}{8 C_S C_d^{\lceil \log_2(128\lambda)\rceil}}.
\end{equation}
Then we have
\[
\rcapa_1(B(x,r)\cap \overline{A}^1,B(x,2r))\le C_1 \rcapa_1(A,B(x,2r)),
\]
where $C_1=C_1(C_d,C_P,\lambda)$ is the constant from Lemma \ref{lem:local boxing inequality}.
\end{lemma}
\begin{proof}
Fix $\eps>0$. By the definition of the variational capacity and the fact that it is an outer capacity, we can pick $u\in N^{1,1}(X)$ with $u\ge 1$ in a neighborhood of $A$,
$u=0$ in $X\setminus B(x,2r)$, and
\[
\rcapa_1(A, B(x,2r))+\eps\ge \int_X g_u\,d\mu \ge \Vert Du\Vert(X),
\]
where $g_u$ is an upper gradient of $u$,
and where the last inequality follows from the fact that Lipschitz functions are dense in
$N^{1,1}(X)$, see e.g. \cite[Theorem 5.1]{BB}.
By using the coarea formula \eqref{eq:coarea}, we find a number $t\in (0,1)$ such that
\begin{equation}\label{eq:estimate for perimeter by means of capacity}
P(\{u>t\},X)\le \Vert Du\Vert(X)\le \rcapa_1(A, B(x,2r))+\eps.
\end{equation}
By the isoperimetric inequality \eqref{eq:isop inequality with zero boundary values},
we have
\begin{align*}
\mu(\{u>t\})
&\le 2 C_S r P(\{u>t\},X)\\
&\le 2 C_S r \rcapa_1(A, B(x,2r))+2C_S r\eps\\
&\overset{\eqref{eq:a priori smallness of capacity}}{\le} 2 C_S r\frac{\mu(B(x,r))}{8 C_S C_d^{\lceil\log_2(128\lambda)\rceil}r}+2C_S r\eps\\
&=  \frac{\mu(B(x,r))}{4 C_d^{\lceil\log_2(128\lambda)\rceil}}+2C_S r\eps.
\end{align*}
Thus by assuming that $\eps\le \mu(B(x,r))/(8C_S C_d^{\lceil\log_2(128\lambda)\rceil}r)$, we have
\[
\frac{\mu(\{u>t\})}{\mu(B(x,2r))}\le \frac{\mu(\{u>t\})}{\mu(B(x,r))}
\le \frac{1}{2 C_d^{\lceil\log_2(128\lambda)\rceil}},
\]
and now by Lemma \ref{lem:local boxing inequality} we have
\begin{equation}\label{eq:estimate for capacity of measure theoretic interior}
\begin{split}
\rcapa_1(B(x, r)\cap I_{\{u>t\}},B(x,2r)) 
&\le C_1 P(\{u>t\},B(x,2r))\\
&\le C_1 P(\{u>t\},X)\\
&\le C_1 \rcapa_1(A, B(x,2r))+C_1\eps.
\end{split}
\end{equation}
by \eqref{eq:estimate for perimeter by means of capacity}.
Again by the definition of the variational capacity, we find a function
$v\in N^{1,1}(X)$ with $v\ge 1$ in $B(x,r)\cap I_{\{u>t\}}$,
$v=0$ in $X\setminus B(x,2r)$, and
\begin{equation}\label{eq:estimate for upper gradient of v}
\int_X g_v\,d\mu\le \rcapa_1(B(x, r)\cap I_{\{u>t\}},B(x,2r))+\eps,
\end{equation}
where $g_v$ is an upper gradient of $v$.
Note that
\begin{equation}\label{eq:first inclusion for A}
A\subset B(x,r)\cap \inte(\{u>t\})\subset B(x,r) \cap I_{\{u>t\}}.
\end{equation}
For a suitable $\mathcal H$-negligible set $N\subset X$, by \eqref{eq:first inclusion for A} and Lemma \ref{lem:fine closure of a set of finite perimeter} we have
\begin{equation}\label{eq:second inclusion for A}
B(x,r)\cap \overline{A}^1\subset B(x,r)\cap \overline{I_{\{u>t\}}}^1\subset B(x,r)\cap (I_{\{u>t\}}\cup \partial^*\{u>t\}\cup N).
\end{equation}
Let $\widetilde{N}\subset X$ be the set of points where $v$ is not $1$-finely continuous.
By Theorem \ref{thm:fine continuity of Newtonian functions},
$\capa_1(\widetilde{N})=0$, and then by \eqref{eq:null sets of Hausdorff measure and capacity}, also
$\mathcal H(\widetilde{N})=0$.
If $y\in B(x,r)\cap \partial^*\{u>t\}\setminus \widetilde{N}$, we deduce $y\in B(x,r)\cap b_1 I_{ \{u>t\}}\setminus \widetilde{N}$ by applying Lemma \ref{lem:inclusion for fine boundaries} with $A=I_{\{u>t\}}$ and noting that $\partial^*\{u>t\}=\partial^*I_{\{u>t\}}$. Thus
necessarily $v(y)\ge 1$. Thus $v\ge 1$ $1$-quasieverywhere in the set
\[
B(x,r)\cap (I_{\{u> t\}}\cup \partial^*\{u> t\}\cup N).
\]
Thus by \eqref{eq:second inclusion for A}, $v\ge 1$ $1$-quasieverywhere in $B(x,r)\cap \overline{A}^1$, and so by  \eqref{eq:estimate for capacity of measure theoretic interior} and \eqref{eq:estimate for upper gradient of v},  we get
\[
\rcapa_1(B(x,r)\cap \overline{A}^1,B(x,2r))\le \int_X g_v\,d\mu
\le C_1 \rcapa_1(A, B(x,2r))+C_1\eps+\eps.
\]
Letting $\eps\to 0$, we obtain the result.
\end{proof}

Now we can give a suitable characterization of the $1$-fine interior. We take the proof almost directly from \cite[Proposition 7.8]{BB-OD}, where it is given for $p>1$.

\begin{proposition}\label{prop:characterization of fine interior}
Let $A\subset X$. Then $\fint A=A\setminus b_1 (X\setminus A)$.
\end{proposition}

\begin{proof}
If $x\in \fint A$, then by definition $X\setminus \fint A$ is
$1$-thin at $x$, and thus so is $X\setminus A$.
Thus $x\in A\setminus b_1(X\setminus A)$.

Conversely, assume that $X\setminus A$ is $1$-thin at $x\in A$, i.e.
\begin{equation}\label{eq:assumption that complement of A is thin}
\lim_{r\to 0} r \frac{\rcapa_1(B(x,r)\setminus A, B(x,2r))}{\mu(B(x,2r))}=0.
\end{equation}
For every $r>0$, let $F_r:=\overline{B(x,r)\setminus A}^1$.
Fix $s>0$.
We show that $F_s$ is $1$-thin at $x$. By \eqref{eq:assumption that complement of A is thin}
it suffices to show that for sufficiently small $0<r\le s$,
\begin{equation}\label{eq:fine closure is controlled for small radii}
\rcapa_1(B(x,r)\cap F_s, B(x,2r))\le C\rcapa_1(B(x,r)\setminus A, B(x,2r))
\end{equation}
for some constant $C>0$.
Note  that for $0<r\le s$, $F_r\cup (X\setminus B(x,r))$ is $1$-finely closed and contains $X\setminus A$, and hence also contains $F_s$. Thus
\[
B(x,r)\cap F_s\subset B(x,r)\cap (F_r\cup (X\setminus B(x,r)))= B(x,r)\cap F_r.
\]
Thus 
\begin{align*}
\rcapa_1(B(x,r)\cap F_s, B(x,2r))
&\le \rcapa_1(B(x,r)\cap F_r,B(x,2r))\\
&\le C_1\rcapa_1(B(x,r)\setminus A, B(x,2r))
\end{align*}
for sufficiently small $0<r\le s$ by Lemma \ref{lem:capacity of fine closure preliminary result}. This establishes \eqref{eq:fine closure is controlled for small radii},
and thus $F_s$ is $1$-thin at $x$.
The set $B(x,s) \setminus F_s$ is $1$-finely open and contained in $A$,
and since $F_s$ is $1$-thin at $x$, the set $(B(x,s)\setminus F_s)\cup\{x\}$ is also
$1$-finely open, and contained in $A$. Thus $(B(x,s)\setminus F_s)\cup \{x\}\subset
\fint A$,
and so $x\in \fint A$.
\end{proof}

Now we can characterize the $1$-fine closure and the $1$-fine boundary in the following way.

\begin{corollary}\label{cor:characterization of fine boundary}
Let $A\subset X$. Then $\overline{A}^1=A\cup b_1 A$ and
\[
\partial^1 A= (A\cap b_1 (X\setminus A))\cup ((X\setminus A)\cap b_1 A).
\]
\end{corollary}

\begin{proof}
Note that $\overline{A}^1=X\setminus \fint(X\setminus A)$. Thus by Proposition \ref{prop:characterization of fine interior},
\[
\overline{A}^1=X\setminus ((X\setminus A)\setminus b_1 A)=A\cup b_1 A.
\]
Then
\begin{align*}
\partial^1 A=\overline{A}^1\setminus \fint A
&=(A\cup b_1 A)\setminus (A\setminus b_1(X\setminus A))\\
&=(A\cup b_1 A)\cap ((X\setminus A)\cup b_1(X\setminus A))\\
&=(A\cap b_1(X\setminus A))\cup (b_1 A\cap (X\setminus A)).
\end{align*}
\end{proof}

The following proposition can be taken directly from \cite[Lemma 4.8]{BBL-WC}, where it is given in the case $p>1$. The proof is also verbatim the same, except that instead of
referring to \cite[Theorem 4.3]{BBL-WC} we refer to Theorem \ref{thm:fine continuity of Newtonian functions}.

\begin{proposition}\label{prop:capacity of fine closure}
Let $A\subset X$. Then $\capa_1(\overline{A}^1)=\capa_1(A)$. If $A\subset D\subset X$, then
\[
\rcapa_1(A,D)=\rcapa_1(\overline{A}^1\cap D,D).
\]
If furthermore $\rcapa_1(A,D)<\infty$, then $\capa_1(\overline{A}^1\setminus \fint D)=0$ and
\[
\rcapa_1(A,D)= \rcapa_1(\overline{A}^1\cap \fint D,\fint D).
\]
\end{proposition}

\section{Necessity of $\mathcal H(\partial^1 I_E)<\infty$}\label{sec:necessity}

In this section we consider the quasicontinuity and fine continuity properties of $\BV$
functions and in particular sets of finite perimeter, and show that every set of finite perimeter $E$ satisfies the condition $\mathcal H(\partial^1 I_E)<\infty$.

The following quasicontinuity-type result is essentially given by \cite[Theorem 1.1]{LaSh},
and later proved in precisely the given form in the below reference.
Recall the definitions of the lower and upper approximate limits $u^{\wedge}$
and $u^{\vee}$ from
\eqref{eq:lower approximate limit} and \eqref{eq:upper approximate limit}.

\begin{theorem}[{\cite[Corollary 4.3]{L2}}]\label{thm:quasicontinuity for sets of finite perimeter}
Let $\Omega\subset X$ be an open set, let $E\subset X$ be a $\mu$-measurable set with
$P(E,\Omega)<\infty$, and let $\eps>0$. Then there exists an open set $G\subset \Omega$ with $\capa_1(G)<\eps$ such that if $y_k\to x$ 
with $y_k,x\in \Omega\setminus G$, then 
\[
\min\{|\ch_{E}^{\wedge}(y_k)-\ch_E^{\wedge}(x)|,\,|\ch_E^{\wedge}(y_k)-\ch_E^{\vee}(x)|\}\to 0
\]
and
\[
\min\{|\ch_E^{\vee}(y_k)-\ch_E^{\wedge}(x)|,\,|\ch_E^{\vee}(y_k)-\ch_E^{\vee}(x)|\}\to 0.
\]
\end{theorem}

Recall that a set $A\subset X$ is \emph{$1$-quasiopen} if for every $\eps>0$ there is an
open set $G\subset X$ with $\capa_1(G)<\eps$ such that $A\cup G$ is open.

\begin{proposition}\label{prop:set of finite perimeter is quasiopen}
Let $\Omega\subset X$ be an open set and let $E\subset X$ be a $\mu$-measurable set with
$P(E,\Omega)<\infty$. Then the sets $I_E\cap\Omega$ and $O_E\cap\Omega$ are $1$-quasiopen.
\end{proposition}

\begin{proof}
Let $\eps>0$.
Let $G\subset \Omega$ be an open set obtained by applying 
Theorem \ref{thm:quasicontinuity for sets of finite perimeter}, so that $\capa_1(G)<\eps$.
Let $V:=(I_E\cap\Omega)\cup G$.
To check that $V$ is open, note first that if $x\in G$, then of course $B(x,r)\subset G\subset V$ for some $r>0$.
If $x\in (I_E\cap\Omega)\setminus G$, then by Theorem \ref{thm:quasicontinuity for sets of finite perimeter} there exists $r>0$ with $B(x,r)\setminus G\subset (I_E\cap\Omega)\setminus G$, and thus
\[
B(x,r)\subset (I_E\cap\Omega)\cup G=V.
\]
Thus $V$ is open.

Since $P(X\setminus E,\Omega)=P(E,\Omega)$ and $O_E=I_{X\setminus E}$, also $O_E\cap\Omega$ is $1$-quasiopen.
\end{proof}

The following fact and its proof are essentially the same as in the case $p>1$,
see \cite[Theorem 1.4]{BBL-CCK}.

\begin{proposition}\label{prop:quasiopen is finely open}
Every $1$-quasiopen set $A\subset X$ is the union of a $1$-finely open set and a $\mathcal H$-negligible set.
\end{proposition}

\begin{proof}
Take open sets $G_j\subset X$ with $\capa_1(G_j)<2^{-j}$, $j\in\N$,
such that each $A\cup G_j$ is an open set.
 By Proposition \ref{prop:capacity of fine closure} we have $\capa_1(\overline{G_j}^1)=\capa_1(G_j)<2^{-j}$. Let $D:=A\cap \bigcap_{j\in\N}\overline{G_j}^1$, so that $\capa_1(D)=0$, and then also $\mathcal H(D)=0$ by \eqref{eq:null sets of Hausdorff measure and capacity}. Then for each $j\in\N$,
$A\setminus \overline{G_j}^1=A\cup G_j\setminus \overline{G_j}^1$
is a $1$-finely open set, since it is the intersection of an open set and a $1$-finely open set. The set
\[
V:=\bigcup_{j\in\N} (A\setminus \overline{G_j}^1)=A\setminus D
\]
is a $1$-finely open set, since it is the union of $1$-finely open sets, and $A=V\cup D$.
\end{proof}

\begin{proposition}\label{prop:sets of finite perimeter are finely open}
Let $\Omega\subset X$ be an open set and let $E\subset X$ be a $\mu$-measurable set with
$P(E,\Omega)<\infty$. Then each of the sets $I_E\cap\Omega$ and $O_E\cap\Omega$ is the union of a $1$-finely open set and a $\mathcal H$-negligible set.
\end{proposition}

\begin{proof}
Combine Propositions \ref{prop:set of finite perimeter is quasiopen}
and \ref{prop:quasiopen is finely open}.
\end{proof}

Recall the characterizations of the $1$-fine closure and the $1$-fine boundary given in Corollary \ref{cor:characterization of fine boundary}.
When we consider measure theoretic interiors, things are simplified somewhat further.

\begin{lemma}\label{lem:fine closure of measure theoretic interior}
For any $\mu$-measurable set $E\subset X$, we have $\overline{I_E}^1=b_1 I_E$ and $\overline{X\setminus I_E}^1=b_1 (X\setminus I_E)$.
\end{lemma}
\begin{proof}
Applying Corollary \ref{cor:characterization of fine boundary}, and the first claim of Lemma \ref{lem:inclusion for fine boundaries} with $A=I_E$, and noting that $I_A=I_E$, we obtain
\[
\overline{I_E}^1=I_E\cup b_1 I_E=b_1 I_E.
\]
Similarly, by applying  Corollary \ref{cor:characterization of fine boundary},
\[
\overline{X\setminus I_E}^1=(X\setminus I_E)\cup b_1(X\setminus I_E)
=O_E\cup \partial^*E\cup b_1(X\setminus I_E)=b_1 (X\setminus I_E),
\]
where the last equality follows from the first claim of Lemma \ref{lem:inclusion for fine boundaries} with $A=X\setminus I_E$, by noting that $I_A=O_E$ and $\partial^*A=\partial^*E$.
\end{proof}

\begin{lemma}\label{lem:inclusion for fine boundaries 2}
For any $\mu$-measurable set $E\subset X$, we have
\[
\partial^1 I_E=b_1 I_E\cap b_1 (X\setminus I_E).
\]
\end{lemma}

\begin{proof}
By the definition of the $1$-fine boundary and Lemma \ref{lem:fine closure of measure theoretic interior}, we have
\[
\partial^1 I_E=\overline{I_E}^1\cap \overline{X\setminus I_E}^1=b_1 I_E\cap b_1(X\setminus I_E).
\]
\end{proof}

Note that by Lemma \ref{lem:inclusion for fine boundaries}, for any $\mu$-measurable set
$E\subset X$ we have $\partial^*E\subset \partial^1 I_E$.
On the other hand, when $E$ is of finite perimeter, these sets almost coincide.
This is the content of the following theorem that is the main result of this section.

\begin{theorem}\label{thm:finite perimeter implies finite boundary}
Let $\Omega\subset X$ be an open set and let $E\subset X$ be a $\mu$-measurable set with
$P(E,\Omega)<\infty$.
Then $\mathcal H((\partial^1 I_E\setminus \partial^*E)\cap\Omega)=0$, and so in particular $\mathcal H(\partial^1 I_E\cap \Omega)<\infty$.
\end{theorem}

\begin{proof}
By Proposition \ref{prop:sets of finite perimeter are finely open} we know that there exist
$1$-finely open sets $A_1,A_2\subset X$ and $\mathcal H$-negligible sets $N_1,N_2\subset X$ 
such that $I_E\cap \Omega=A_1\cup N_1$ and $O_E\cap \Omega=A_2\cup N_2$.
Take $x\in \Omega\setminus (\partial^*E\cup N_1\cup N_2)$. Thus $x\in A_1\cup A_2$. Suppose $x\in A_1$. Then
\[
r\frac{\rcapa_1(B(x,r)\setminus I_E,B(x,2r))}{\mu(B(x,r))}
\le r\frac{\rcapa_1(B(x,r)\setminus A_1,B(x,2r))}{\mu(B(x,r))}
\to 0
\]
as $r\to 0$, since $A_1$ is a $1$-finely open set. Thus $x\notin b_1 (X\setminus I_E)$, and by Lemma \ref{lem:inclusion for fine boundaries 2}, $x\notin \partial^1 I_E$. If $x\in A_2$, similarly
\begin{align*}
r\frac{\rcapa_1(B(x,r)\cap I_E,B(x,2r))}{\mu(B(x,r))}
&\le r\frac{\rcapa_1(B(x,r)\setminus O_E,B(x,2r))}{\mu(B(x,r))}\\
&\le r\frac{\rcapa_1(B(x,r)\setminus A_2,B(x,2r))}{\mu(B(x,r))}
\to 0
\end{align*}
as $r\to 0$, since $A_2$ is a $1$-finely open set. Thus $x\notin b_1 I_E$, and so $x\notin \partial^1 I_E$.
In conclusion, $(\partial^1 I_E\setminus \partial^* E)\cap \Omega\subset N_1\cup N_2$, so that
$\mathcal H((\partial^1 I_E\setminus \partial^*E)\cap\Omega)=0$. By \eqref{eq:def of theta} we know that $\mathcal H(\partial^*E\cap \Omega)<\infty$, so now also $\mathcal H(\partial^1 I_E\cap \Omega)<\infty$.
\end{proof}

Note that we have the following.

\begin{lemma}\label{lem:coincidence of fine boundaries}
For any $\mu$-measurable set $E\subset X$, we have $\partial^1 I_E=\partial^1 O_E$.
\end{lemma}
\begin{proof}
Fix $x\in \partial^1 O_E$. Then $x\in b_1 O_E$ by Lemma \ref{lem:inclusion for fine boundaries 2}, and thus $x\in b_1 (X\setminus I_E)$.
We need to show that also $x\in b_1 I_E$.
Note that for any $r>0$, $B(x,r)\cap b_1 I_E\subset b_1(B(x,r)\cap I_E)$.
By using Proposition \ref{prop:capacity of fine closure} and the first claim of Lemma \ref{lem:inclusion for fine boundaries} with $A=I_E$ (note that $I_A=I_E$), we obtain that
\begin{align*}
&\limsup_{r\to 0}r \frac{\rcapa_1(B(x,r)\cap I_E,B(x,2r))}{\mu(B(x,r))}\\
&\qquad\qquad \ge \limsup_{r\to 0} r \frac{\rcapa_1(b_1 (B(x,r)\cap I_E),B(x,2r))}{\mu(B(x,r))}\\
&\qquad\qquad \ge \limsup_{r\to 0} r \frac{\rcapa_1(B(x,r)\cap b_1 I_E,B(x,2r))}{\mu(B(x,r))}\\
&\qquad\qquad  \ge \limsup_{r\to 0} r \frac{\rcapa_1(B(x,r)\cap (I_E\cup \partial^*I_E),B(x,2r))}{\mu(B(x,r))}\\
&\qquad\qquad  = \limsup_{r\to 0} r \frac{\rcapa_1(B(x,r)\setminus O_E,B(x,2r))}{\mu(B(x,r))}>0,
\end{align*}
since $x\in b_1 (X\setminus O_E)$ by Lemma \ref{lem:inclusion for fine boundaries 2}. Thus $x\in \partial^1 I_E$.
\end{proof}

\begin{remark}
Despite the above lemma, $\partial^1 I_E$ seems in some way a strange set to consider,
since  we seem to obtain it by first taking an open set in the measure topology, and by then taking the boundary in a different topology, namely the $1$-fine topology.
To clarify the issue somewhat, let us see what happens if we define the measure topology in a more axiomatic way than is used in defining $I_E$, $O_E$, and $\partial^*E$. We say that a set $U\subset X$ is $0$-finely open
if
\[
\lim_{r\to 0}\frac{\mu(B(x,r)\setminus U)}{\mu(B(x,r))}=0
\]
for every $x\in U$, that is, $U\subset I_U$.
Then as in Definition \ref{def:1 fine topology}, for any $G\subset X$ we can
 define the $0$-fine interior $\0fint G$, the $0$-fine closure $\overline{G}^0$, and the $0$-fine boundary $\partial^0 G:=\overline{G}^0\setminus \0fint G$. Moreover, let $b_0 G$ be the set of points where the density of $G$ is \emph{not} zero (somewhat confusingly), i.e. $b_0 G=X\setminus O_G$. Analogously to Proposition \ref{prop:characterization of fine interior} and Corollary \ref{cor:characterization of fine boundary}, we can then show that
\[
\0fint G= G\setminus b_0(X\setminus G),\qquad \overline{G}^0=G\cup b_0 G,
\]
and
\[
\partial^0 G=(G\cap b_0(X\setminus G))\cup ((X\setminus G)\cap b_0 G).
\]
Then as in Lemma \ref{lem:inclusion for fine boundaries 2}, we have at least for any $\mu$-measurable set $E\subset X$
\[
\partial^0 I_E=b_0 I_E\cap b_0 (X\setminus I_E).
\]
But this is exactly $\partial^*I_E=\partial^*E$. Thus the measure theoretic boundary
$\partial^*E$ is the boundary of $I_E$ in the $0$-fine topology
(and not of $E$).
Moreover, $I_E$ is not the same as $\0fint E=E\cap I_E$, so perhaps $I_E$ should be
viewed only as a measure theoretic concept, and not a topological one.
Now the conclusion of Theorem \ref{thm:finite perimeter implies finite boundary} can be reformulated in the more symmetric fashion $\mathcal H(\partial^1 I_E\setminus \partial^0 I_E)=0$.
\end{remark}

\begin{example}
Let $X=\R^2$ (unweighted), and consider the slit disk
\[
E=B(0,1)\setminus \{x=(x_1,x_2):\,x_1>0,\,x_2=0\}.
\]
Note that the $1$-dimensional Hausdorff measure $\mathcal H^1$ is comparable to the codimension one Hausdorff measure $\mathcal H$.
Now $\partial^*E=\partial B(0,1)$ and $P(E,\R^2)=\mathcal H^1(\partial^*E)=2\pi$.
From Corollary \ref{cor:characterization of fine boundary} it follows that
$\partial^1 E=\partial E$,
so that $\partial^1 E\setminus \partial^*E$ consists of the slit, and so $\mathcal H(\partial^1 E\setminus \partial^*E)>0$. Similarly,
$\fint E=E$, and thus also $\mathcal H(\partial^1 \fint E\setminus \partial^*E)>0$.

Thus in Theorem \ref{thm:finite perimeter implies finite boundary}, we cannot replace $I_E$ by either $E$ or $\fint E$.
\end{example}

\begin{example}\label{ex:real line}
Let $X=\R$ (unweighted). For any $x\in \R$, $r>0$, and $A\subset \R$
with $B(x,r)\cap A\neq \emptyset$ we have
\[
\rcapa_1(B(x,r)\cap A,B(x,2r))\ge 2.
\]
Thus we find that
the metric topology and the $1$-fine topology coincide, and so
for any $A\subset \R$ we have $\partial^1 A=\partial A$. For a set of finite perimeter
$E\subset \R$,
the result $\mathcal H(\partial^1 I_E\setminus \partial^*E)=0$ thus implies that
$\partial I_E= \partial^*E$, since $\mathcal H$ is now (comparable to) the counting measure.
Thus $I_E$ is a \emph{good representative} of $E$, a well-known result on the real line, see e.g. \cite[Sections 3.2 \& 3.5]{AFP}.

On the other hand, if we define
\[
E:=\R\setminus \bigcup_{j\in\N} [2^{-j},2^{-j}+ 2^{-2j}],
\]
then clearly $P(E,\R)=\infty$ and $0\in I_E$. If $U\supset I_E$ is an open set, then $U$ contains a neighborhood of the origin and  thus $U\setminus I_E\neq \emptyset$. Thus $\capa_1(U\setminus I_E)\ge 2$, since every point has $1$-capacity $2$. Thus $I_E$
 is not a $1$-quasiopen set.
Moreover, $\mathcal H(\partial^1 I_E\setminus \partial^*E)=\mathcal H(\{0\})>0$, so the
conclusions of Proposition \ref{prop:set of finite perimeter is quasiopen} and Theorem \ref{thm:finite perimeter implies finite boundary} do not
necessarily hold unless $E$ is of finite perimeter. 
\end{example}

Given a set of finite perimeter $E\subset X$, if we denote by $\Sigma_{\gamma}E$ the subset of $\partial^* E$ where \eqref{eq:definition of gamma} holds, we know that $\mathcal H(\partial^*E\setminus \Sigma_{\gamma}E)=0$. Theorem \ref{thm:finite perimeter implies finite boundary} then shows that the difference between $\partial^* E$ and the a priori larger set $\partial^1 I_E$
is also $\mathcal H$-negligible. In conclusion, the boundary of a set of finite perimeter is quite regular in the sense that all of these sets almost coincide.

\section{Sufficiency  of $\mathcal H(\partial^1 I_E)<\infty$}\label{sec:sufficiency}

In this section we prove that the condition $\mathcal H(\partial^1 I_E)<\infty$ is also sufficient for $E$ to be of finite perimeter.

\begin{theorem}\label{thm:sufficiency}
Let $\Omega\subset X$ be an open set, let $E\subset X$ be a $\mu$-measurable set, and assume that $\mathcal H(\partial^1 I_E\cap \Omega)<\infty$. Then $P(E,\Omega)<\infty$.
\end{theorem}

In proving this result, we will need to study the concept of \emph{capacitary potential}
in the case $p=1$.
Given an open set $U\subset X$ and an arbitrary set $H\subset U$, we define
the variational $\BV$-capacity by
\[
\rcapa_{\BV}(H,U):=\inf\Vert Du\Vert(X),
\]
where the infimum is taken over functions $u\in L^1_{\loc}(X)$ with $u\ge 1$ in a neighborhood of
$H$ and $u=0$ in $X\setminus U$.

Note that by the coarea formula \eqref{eq:coarea}, we know that in fact
\[
\rcapa_{\BV}(H,U)=\inf P(D,X),
\]
where the infimum is taken over $\mu$-measurable sets $D\subset U$ containing a neighborhood of $H$. Now we give a new characterization of the variational $\BV$-capacity.
First we take note of the following lemmas.

\begin{lemma}[{\cite[Lemma 3.1]{L}}]\label{lem:perimeter of G}
For any $G\subset X$, we can find an open set $V\supset G$ with
$\capa_1(V)\le C_2\capa_1(G)$ and $P(V,X)\le C_2\capa_1(G)$,
where $C_2=C_2(C_d,C_P,\lambda)$.
\end{lemma}

\begin{lemma}[{\cite[Lemma 3.9]{L2}}]\label{lem:absolute continuity of variation wrt capacity}
Let $\Omega\subset X$ be an open set and
let $u\in L^1_{\loc}(\Omega)$ with $\Vert Du\Vert(\Omega)<\infty$. Then for every $\eps>0$ there exists $\delta>0$ such that if $A\subset \Omega$ with $\capa_1(A)<\delta$, then $\Vert Du\Vert(A)<\eps$.
\end{lemma}

\begin{proposition}\label{prop:characterization of capacity}
Given an open set $U\subset X$ and  $H\subset U$, we have
\[
\rcapa_{\BV}(H,U)=\inf P(D,X),
\]
where the infimum is taken over $\mu$-measurable sets $D\subset U$ with $H\subset I_D$.
\end{proposition}
\begin{proof}
One inequality is clear. To prove the opposite inequality, we can assume that there
exists a $\mu$-measurable set $D\subset U$ with $H\subset I_D$ and $P(D,X)<\infty$.
By applying the coarea formula \eqref{eq:coarea}, we find open sets $U_i\Subset U$ with
 $U=\bigcup_{i\in\N}U_i$, and
$P(U_i,X)<\infty$ for each $i\in\N$.
Fix $\eps>0$, and then fix $i\in\N$.
By Lemma \ref{lem:absolute continuity of variation wrt capacity} there exists $\delta\in (0,\eps)$ such that if $A\subset X$ with $\capa_1(A)<\delta$, then
\begin{equation}\label{eq:small perimeter of Ui in small capacity set}
P(U_i,A)<\frac{2^{-i-1}\alpha \eps}{C_d};
\end{equation}
recall the constant $\alpha$ from \eqref{eq:def of theta}.
By Proposition \ref{prop:set of finite perimeter is quasiopen} we find a set
$G_i\subset X$ with
\[
\capa_1(G_i)<\frac{2^{-i-1}\alpha\delta}{C_d C_2}
\]
such that $(I_D\cap U_i)\cup G_i$ is an open set.
By Lemma \ref{lem:perimeter of G} we find an open set $V_i\supset G_i$ with
$\capa_1(V_i)< \delta$ and $P(V_i,X)< 2^{-i-1}\alpha \eps/C_d$.
By Proposition \ref{prop:capacity of fine closure}, also $\capa_1(I_{V_i})< \delta$.
Clearly $\partial^*(U_i\cap V_i)\subset (\partial^*U_i\cap I_{V_i})\cup \partial^*V_i$.
By \eqref{eq:sets of finite perimeter form an algebra}, $P(U_i\cup V_i,X)<\infty$, and
then by \eqref{eq:def of theta},
\begin{equation}\label{eq:perimeter of Ui cap Vi}
\begin{split}
P(U_i\cap V_i,X)&\le C_d\mathcal H(\partial^*(U_i\cap V_i))\\
&\le C_d(\mathcal H(\partial^*U_i\cap I_{V_i})+
\mathcal H(\partial^*V_i))\\
&\le C_d\alpha^{-1}( P(U_i,I_{V_i})+P(V_i,X))\\
&\overset{\eqref{eq:small perimeter of Ui in small capacity set}}{\le} C_d\alpha^{-1}(2^{-i-1}\alpha \eps/C_d+2^{-i-1}\alpha \eps/C_d)\\
&= 2^{-i} \eps.
\end{split}
\end{equation}
This can be done for each $i\in\N$, and then the set
\[
(I_D\cap U)\cup \bigcup_{i\in\N} (U_i\cap V_i)
=\bigcup_{i\in\N} (I_D\cup V_i)\cap U_i=\bigcup_{i\in\N} ((I_D\cap U_i)\cup V_i)\cap U_i
\]
contains $H$, is contained in $U$, and is an open set since each $(I_D\cap U_i)\cup V_i$ and each $U_i$ is an open set.
Moreover, by the fact that $\mu((I_D\cap U)\Delta D)=0$, and by the lower semicontinuity
and subadditivity \eqref{eq:sets of finite perimeter form an algebra} of perimeter,
\begin{align*}
P\left((I_D\cap U)\cup \bigcup_{i\in\N} (U_i\cap V_i),X\right)
&= P\left(D\cup \bigcup_{i\in\N} (U_i\cap V_i),X\right)\\
&\le P(D,X)+\sum_{i\in\N}P(U_i\cap V_i,X)\\
&\le P(D,X)+\eps
\end{align*}
by \eqref{eq:perimeter of Ui cap Vi}.
Since $\eps>0$ was arbitrary,
we obtain the result.
\end{proof}

\begin{remark}
The above proposition essentially states that the $\BV$-capacity turns out to be an
outer capacity even if we do not define it as such. This is similar to
\cite[Theorem 4.1]{BB-VC}, where it is shown that the variational capacity $\rcapa_p$
is an outer capacity under very weak assumptions.

Moreover, Proposition \ref{prop:characterization of capacity} and another application of the coarea formula give
\begin{equation}\label{eq:BV lower limit characterization of BV capacity}
\rcapa_{\BV}(H,U)=\inf\Vert Du\Vert(X),
\end{equation}
where the infimum is taken over functions $u\in L^1_{\loc}(X)$ with $u^{\wedge}\ge 1$ in
$H$ and $u=0$ in $X\setminus U$ (more precisely, $u=0$ $\mu$-almost everywhere in $X\setminus U$;
recall that we understand $\BV$ functions to be $\mu$-equivalence classes).
\end{remark}

Following the definitions and terminology used in the case $p>1$, we give the following definition.

\begin{definition}
Given an open set $U\subset X$ and  $H\subset U$, we say that a $\mu$-measurable set
$D\subset U$ is a $1$-capacitary potential for $H$ in $U$ if
$H\subset I_D$ and
\[
P(D,X)=\rcapa_{\BV}(H,U)<\infty.
\]
\end{definition}

Of course, Proposition \ref{prop:characterization of capacity} guarantees
that this definition makes sense.
Almost any example on the real line shows that a $1$-capacitary potential is not unique, contrary to the case $p>1$. However, we have the following existence result.

\begin{proposition}\label{prop:existence of 1capacitary potential}
Let $U\subset X$ be an open set and let $H\subset X$ with $I_H\subset U$.
If $\rcapa_{\BV}(I_H,U)<\infty$, then a $1$-capacitary potential for $I_H$ in $ U$ exists.
\end{proposition}

Note that here we do not even need $H$ to be
$\mu$-measurable. However, it can be
shown that $I_H$, $O_H$, and $\partial^*H$ are still Borel sets.

\begin{proof}
Take a sequence of $\mu$-measurable sets $D_i\subset  U$ with $I_H\subset I_{D_i}$ and
\[
P(D_i,X)\to \rcapa_{\BV}(I_H,U).
\]
By the weak compactness of $\BV$ functions, see \cite[Theorem 3.7]{M}, there exists $D\subset U$ such that by passing to a subsequence (not relabeled),
$\ch_{D_i}\to \ch_D$ in $L^1_{\loc}(X)$. By the lower semicontinuity of perimeter with respect to convergence in $L^1_{\loc}(X)$, we have
\[
P(D,X)\le \liminf_{i\to\infty} P(D_i,X)=\rcapa_{\BV}(I_H,U).
\]
Fix $x\in X$ and $R>0$. Since also $\ch_{X\setminus D_i}\to \ch_{X\setminus D}$ in $L^1_{\loc}(X)$, we have
\[
\mu(B(x,R)\cap I_H\setminus D)=\int_{B(x,R)\cap I_H}\ch_{X\setminus D}\,d\mu
=\lim_{i\to\infty}\int_{B(x,R)\cap I_H}\ch_{X\setminus D_i}\,d\mu=0
\]
since $\mu(I_H\setminus D_i)=0$ for all $i\in\N$ (by the Lebesgue differentiation theorem). By letting $R\to\infty$, we get
$\mu(I_H\setminus D)=0$, and so $I_H\subset I_D$.
Then $D$ is a $1$-capacitary potential for $I_H$ in $U$.
\end{proof}

The crux of the proof of Theorem \ref{thm:sufficiency} is obtaining the following \emph{Choquet}-type property in the case $p=1$.
For this, we combine ideas from the proof of the case $p>1$ given in \cite[Theorem 7.1]{BBL-CCK} with methods of metric space $\BV$ theory, see especially
\cite[Theorem 5.3]{A1}.

\begin{proposition}\label{prop:Choquet type property}
Let $A\subset X$ be a $\mu$-measurable set and let $\eps>0$. Then there exists an open set
$U\subset X$ with $U\cup b_1 A=X$ and
\[
\capa_1(U\cap \overline{I_A}^1)<\eps.
\]
\end{proposition}
\begin{proof}
Let $\{B(x_j,r_j)\}_{j\in\N}$ be a covering of $X$ by balls such that every point is covered by arbitrarily small balls; this is possible since the space is separable.
We can assume that $r_j<\diam(X)/8$ for all $j\in\N$.
For each $j\in\N$, by \cite[Lemma 6.2]{KKST} there exists $s\in [r_j,2r_j]$ such that
\[
P(B(x_j,s))\le 2C_d\frac{\mu(B(x_j,s))}{s}<\infty.
\]
In particular, $\rcapa_{\BV}(I_{B(x_j,r_j)\cap A},B(x_j,2r_j))<\infty$, and so by Proposition \ref{prop:existence of 1capacitary potential}, for each $j\in\N$ we can let  $D_j\subset X$ be a $1$-capacitary potential for $I_{B(x_j,r_j)\cap A}$ in $B(x_j,2r_j)$.
Then $B(x_j,r_j)\cap I_A\subset I_{D_j}$, and thus
\begin{equation}\label{eq:solutions do not intersect measure theoretic interior}
B(x_j,r_j)\cap O_{D_j}\cap I_A=\emptyset\qquad\textrm{for each }j\in\N.
\end{equation}

\paragraph{Claim:} We have $\mathcal H\left(X \setminus \big(b_1 A\cup \bigcup_{j\in\N}(B(x_j,r_j)\cap O_{D_j})\big)\right)=0$.

\paragraph{Proof:}For any set of finite perimeter $E\subset X$, denote by $\Sigma_{\gamma} E$ the subset of $\partial^*E$ where \eqref{eq:definition of gamma} holds. Then let
\[
N:=\bigcup_{j\in\N}\partial^*D_j\setminus \Sigma_{\gamma}D_j,
\]
so that $\mathcal H(N)=0$.
Fix $x\in X\setminus (b_1 A\cup N)$.
If for some $j\in\N$ we have $x\in B(x_j,r_j)$ and
\[
\liminf_{r\to 0}\frac{\mu(B(x,r)\cap D_j)}{\mu(B(x,r))}<\gamma,
\]
then $x\in B(x_j,r_j)\cap O_{D_j}$.
We assume that
\begin{equation}\label{eq:solution has density at least gamma}
\liminf_{r\to 0}\frac{\mu(B(x,r)\cap D_j)}{\mu(B(x,r))}\ge \gamma
\end{equation}
for all $j\in\N$ with $x\in B(x_j,r_j)$, and derive a contradiction,
thus proving the claim.
Define $\delta>0$ by
\[
\delta:=\min\left\{\frac{\gamma}{2},\frac{1}{(3C_{SP})^{Q} 
C_d^{(1+\lceil\log_2(4\lambda)\rceil)(Q+1)}} \right\}.
\]
Here $\lceil a \rceil$ is the smallest integer at least $a\in\R$,
and $C_{SP}$ is the constant appearing in the relative isoperimetric inequality \eqref{eq:relative isoperimetric inequality}.
Fix $j\in\N$ such that $B(x_j,r_j)\ni x$, and $r_j>0$ is so small that
\begin{equation}\label{eq:smallness of capacity around x}
r\frac{\rcapa_1(B(x,r)\cap A,B(x,2r))}{\mu(B(x,r))}<\frac{\delta}{5 C_d^3 C_S^2}
\end{equation}
for all $r\in (0,2r_j]$, where $C_S$ is the constant from the Sobolev inequality \eqref{eq:sobolev inequality}. For any such $r$, using the definition of the variational capacity, we find $u_r\in N^{1,1}(X)$ with $u_r\ge 1$ in $B(x,r)\cap A$, $u=0$ in $X\setminus B(x,2r)$, and 
\[
\frac{\delta}{5 C_d^3 C_S^2}\ge \frac{r}{\mu(B(x,r))}\int_X g_{u_r}\,d\mu\ge 
\frac{r}{\mu(B(x,r))}\Vert Du_r\Vert(X),
\]
where $g_{u_r}$ is an upper gradient of $u_r$,
and where the last inequality follows from the fact that Lipschitz functions are dense in
$N^{1,1}(X)$, see e.g. \cite[Theorem 5.1]{BB}.
By using the coarea formula \eqref{eq:coarea}, we find a number $t\in (0,1)$ such that
\[
\frac{r}{\mu(B(x,r))} P(\{u_r>t\},X)\le \frac{\delta}{5 C_d^3 C_S^2}.
\]
In conclusion, for any $r\in (0,2r_j]$ there exists a set $D\subset B(x,2r)$ covering $B(x,r)\cap A$ such that
\begin{equation}\label{eq:perimeter of Dj 1}
\frac{r}{\mu(B(x,r))} P(D,X)\le \frac{\delta}{5 C_d^3 C_S^2}.
\end{equation}
On the other hand, by using \cite[Lemma 11.22]{BB}, the doubling property of the measure, and \eqref{eq:smallness of capacity around x} with $r=2r_j$, we obtain
\begin{equation}\label{eq:smallness of capacity in Bj}
\begin{split}
r_j\frac{\rcapa_1(B(x_j,r_j)\cap A,B(x_j,2r_j))}{\mu(B(x_j,r_j))}
&\le 5C_S  r_j\frac{\rcapa_1(B(x_j,r_j)\cap A,B(x_j,4r_j))}{\mu(B(x_j,r_j))}\\
&\le 5C_S C_d^2 r_j\frac{\rcapa_1(B(x,2r_j)\cap A,B(x,4r_j))}{\mu(B(x,2r_j))}\\
&< \frac{\delta}{2C_d C_S}.
\end{split}
\end{equation}
Then as above, we conclude that there exists a set $\widetilde{D}\subset B(x_j,2r_j)$ with
$\widetilde{D}\supset B(x_j,r_j)\cap A$ and
\[
\frac{r_j}{\mu(B(x_j,r_j))} P(\widetilde{D},X)< \frac{\delta}{2C_d C_S}.
\]
Note that $I_{B(x_j,r_j)\cap A}\subset I_{\widetilde{D}}$, and so
$\widetilde{D}$ is admissible for
$\rcapa_{\BV}(I_{B(x_j,r_j)\cap A},\allowbreak B(x_j,2r_j))$.
Then since  $D_j\subset X$ is a $1$-capacitary potential for $I_{B(x_j,r_j)\cap A}$ in $B(x_j,2r_j)$, we necessarily have also
\begin{equation}\label{eq:perimeter of Dj 2}
\frac{r_j}{\mu(B(x_j,r_j))} P(D_j,X)<  \frac{\delta}{2C_d C_S}.
\end{equation}
By \eqref{eq:solution has density at least gamma} we have
\[
\liminf_{r\to 0}\frac{\mu(B(x,r)\cap D_j)}{\mu(B(x,r))}> \delta.
\]
However,
\begin{equation}\label{eq:smalless of the measure of Aj}
\frac{\mu(B(x,r)\cap D_j)}{\mu(B(x,r))}<\delta\qquad \textrm{for all }r\ge r_j;
\end{equation}
this can be seen as follows. Let $r\ge r_j$, and
note that $D_j\subset B(x_j,2r_j)$.
By the isoperimetric inequality \eqref{eq:isop inequality with zero boundary values}
and \eqref{eq:perimeter of Dj 2},
\begin{align*}
\mu(B(x,r)\cap D_j)
\le \mu(D_j)
&\le 2C_S r_j P(D_j,X)\\
&< 2C_S r_j \frac{\delta \mu(B(x_j,r_j))}{2C_d C_S r_j}
=  \frac{\delta \mu(B(x_j,r_j))}{C_d}.
\end{align*}
But now $B(x_j,r_j)\subset B(x,2r)$, and so by the doubling property of the measure,
\[
\mu(B(x_j,r_j))\le C_d \mu(B(x,r)),
\]
establishing \eqref{eq:smalless of the measure of Aj}.
Thus we can define $R\in (0, r_j]$ by
\[
R:=\inf \left\{r>0:\, \frac{\mu(B(x,r)\cap D_j)}{\mu(B(x,r))}< \delta\right\}.
\]
Clearly
\begin{equation}\label{eq:smallness of Dj in the boundary case}
\frac{\mu(B(x,R)\cap D_j)}{\mu(B(x,R))}< C_d\delta.
\end{equation}
It follows that
\begin{equation}\label{eq:smallness of Dj in one over lambda ball}
\frac{\mu(B(x,R/(4\lambda))\cap D_j)}{\mu(B(x,R/(4\lambda)))}\le C_d^{1+\lceil\log_2 (4\lambda)\rceil}\delta
\le \frac{1}{2}.
\end{equation}
Thus by the definition of $R$, the relative isoperimetric inequality \eqref{eq:relative isoperimetric inequality},
and \eqref{eq:smallness of Dj in one over lambda ball}, we obtain
\begin{align*}
\delta \frac{\mu(B(x,R))}{C_d^{\lceil\log_2 (4\lambda)\rceil}}
&\le \delta \mu(B(x,R/(4\lambda)))\\
&\le \mu(B(x,R/(4\lambda))\cap D_j)\\
&\le 2C_{SP}\frac{R}{4\lambda}\left(\frac{\mu(B(x,R/(4\lambda))\cap D_j)}{\mu(B(x,R/(4\lambda)))}\right)^{1/Q} P(D_j,B(x,R/2))\\
&\le C_{SP}\frac{R}{2\lambda}\left(C_d^{1+\lceil\log_2 (4\lambda)\rceil}\delta\right)^{1/Q} P(D_j,B(x,R/2)).
\end{align*}
By the choice of $\delta$, this implies that
\begin{equation}\label{eq:perimeter of Dj in ball to be modified}
P(D_j,B(x,R/2))\ge 6 C_d \delta \frac{\mu(B(x,R))}{R}.
\end{equation}
But by \eqref{eq:perimeter of Dj 1}, we find a $\mu$-measurable set $D\subset B(x,2R)$ covering $B(x,R)\cap A$ such that
\begin{equation}\label{eq:perimeter of the better solution}
P(D,X)\le \frac{\delta}{5C_d^3 C_S^2}\frac{\mu(B(x,R))}{R}.
\end{equation}
Then by the isoperimetric inequality \eqref{eq:isop inequality with zero boundary values}, we have
\begin{equation}\label{eq:measure of D}
\mu(D)\le 2C_S R P(D,X)\le  \frac{\delta\mu(B(x,R))}{C_d^3 C_S}.
\end{equation}
Take $\eta\in \Lip_c(X)$ with $0\le \eta \le 1$, $\eta=1$ in $B(x,R/2)$, $\eta=0$ in
$X\setminus B(x,R)$, and with an upper gradient $g_{\eta}\le 2/R$.
Let
\[
w:=\eta \ch_{I_D} +(1-\eta)\ch_{I_{D_j}}.
\]
Then by a Leibniz rule, see \cite[Lemma 3.2]{HKLS},
\begin{align*}
&\Vert Dw\Vert(X)\\
&\ \ \le P(D,B(x,R))+P(D_j,X\setminus B(x,R/2))+
\int_{D\cup D_j}g_\eta\,d\mu\\
&\ \ \overset{\eqref{eq:perimeter of the better solution}}{\le} \frac{\delta\mu(B(x,R))}{5C_d^3 C_S^2 R}+P(D_j,X\setminus B(x,R/2))
+\frac{2\mu(B(x,R)\cap(D\cup D_j))}{R}\\
&\ \ \le \frac{\delta\mu(B(x,R))}{5C_d^3 C_S^2 R}+P(D_j,X\setminus B(x,R/2))
+\frac{2\delta}{R}\left(\frac{1}{C_d^3 C_S}+ C_d\right)\mu(B(x,R))
\end{align*}
by \eqref{eq:smallness of Dj in the boundary case} and \eqref{eq:measure of D}.
Thus by \eqref{eq:perimeter of Dj in ball to be modified}
\begin{equation}\label{eq:contradictory inequality}
\begin{split}
&\Vert Dw\Vert(X)-P(D_j,X)\\
&\qquad \le \frac{\delta\mu(B(x,R))}{5C_d^3 C_S^2 R}
+\frac{2\delta}{R}\left(\frac{1}{C_d^3 C_S}+ C_d\right)\mu(B(x,R))
-P(D_j,B(x,R/2))\\
&\qquad \le \frac{\delta\mu(B(x,R))}{5C_d^3 C_S^2 R}
+\frac{2\delta}{R}\left(\frac{1}{C_d^3 C_S}+ C_d\right)\mu(B(x,R))-6 C_d \delta \frac{\mu(B(x,R))}{R}\\
&\qquad \le -C_d\delta \frac{\mu(B(x,R))}{R}<0.
\end{split}
\end{equation}
Recall that $I_{B(x_j,r_j)\cap A}\subset I_{D_j}$.
Since $B(x,R)\cap A\subset D$, also $I_{B(x_j,r_j)\cap A}\cap B(x,R)\subset I_{D}$. Thus we have $I_{B(x_j,r_j)\cap A}\subset \{w=1\}$,
and so
$I_{B(x_j,r_j)\cap A}\subset \{w^{\wedge}=1\}$.
Since also $B(x,R)\subset B(x_j,2r_j)$, we have $w=0$ in $X\setminus B(x_j,2r_j)$.
But now \eqref{eq:contradictory inequality} is a contradiction by \eqref{eq:BV lower limit characterization of BV capacity}, since $D_j$ is a $1$-capacitary potential for $I_{B(x_j,r_j)\cap A}$ in $B(x_j,2r_j)$. Thus the claim is proved.\\

By the claim, we have
\begin{equation}\label{eq:thick points and solutions and N cover whole space}
b_1 A\cup \bigcup_{j\in\N} \left(B(x_j,r_j)\cap O_{D_j}\right)\cup N=X,
\end{equation}
where $\mathcal H(N)=0$.
Since each $B(x_j,r_j)\cap O_{D_j}$ is $1$-quasiopen by Proposition \ref{prop:set of finite perimeter is quasiopen}, there exist open sets $G_j\subset X$ with
$\capa_1(G_j)<2^{-j-1}\eps$ such that each
\[
U_j:=(B(x_j,r_j)\cap O_{D_j})\cup G_j
\]
is open.
For the exceptional set $N$ with $\mathcal H(N)=0$,
we have also $\capa_1(N)=0$ by \eqref{eq:null sets of Hausdorff measure and capacity}, and then since $\capa_1$ is an outer capacity, we find an open set $V\supset N$ with $\capa_1(V)<\eps/2$.
Let $U:=\bigcup_{j\in\N}U_j\cup V$. 
Then by \eqref{eq:thick points and solutions and N cover whole space},
$b_1 A\cup U=X$.
By \eqref{eq:solutions do not intersect measure theoretic interior},
\[
\capa_1(U\cap I_A)=\capa_1\left(\Big(\bigcup_{j\in\N}G_j\cup V\Big)\cap I_A\right)\le \sum_{j\in\N}\capa_1(G_j)+\capa_1(V)<\eps.
\]
Note that since $U$ is open,
\[
U\cap b_1 I_A\subset b_1(U\cap I_A).
\]
Moreover, $\overline{I_A}^1=b_1 I_A$ by Lemma \ref{lem:fine closure of measure theoretic interior}.
Combining these,
\[
U\cap \overline{I_A}^1= U\cap b_1 I_A \subset b_1 (U\cap I_A) \subset \overline{U\cap I_A}^1.
\]
Now by Proposition \ref{prop:capacity of fine closure},
\[
\capa_1(U\cap \overline{I_A}^1) \le \capa_1(\overline{U\cap I_A}^1)
 =\capa_1(U\cap I_A)< \eps.
\]
\end{proof}

\begin{remark}
The Claim whose proof took up the bulk of the proof of Proposition \ref{prop:Choquet type property} can be seen as a weak analog of the \emph{Cartan property} that holds in the case $p>1$, see \cite[Theorem 1.1]{BBL-CCK}.
\end{remark}

A more exact analog of the Choquet property that is known to hold for $p>1$, see 
\cite[Theorem 7.1]{BBL-CCK}, is given in the following open problem.

\begin{openproblem}
Let $A\subset X$ and $\eps>0$.  Can we find
an open set $U\subset X$ such that $U\cup b_1 A=X$ and $\capa_1(U\cap A)<\eps$?
\end{openproblem}

\begin{proposition}\label{prop:complement of fine closure is quasiopen}
For any $\mu$-measurable set $E\subset X$, $X\setminus \overline{I_E}^1$ is a $1$-quasiopen set.
\end{proposition}
\begin{proof}
Fix $\eps>0$.
By applying Proposition \ref{prop:Choquet type property} with $A=I_E$, and noting that $I_A=I_E$, we find an open set $U\supset X\setminus b_1 I_E\supset X\setminus \overline{I_E}^1$ such that
\[
\capa_1(U\cap \overline{I_E}^1)<\eps.
\]
Since $\capa_1$ is an outer capacity,
we find an open set $G\supset U\cap \overline{I_E}^1$ with $\capa_1(G)<\eps$.
Thus the set
\[
(X\setminus \overline{I_E}^1)\cup G=U\cup G
\]
is open, and thus $X\setminus \overline{I_E}^1$ is a $1$-quasiopen set.
\end{proof}

\begin{openproblem}
Is every $1$-finely open set $1$-quasiopen?
\end{openproblem}

Note that Proposition \ref{prop:complement of fine closure is quasiopen} gives a partial
result in this direction, since $X\setminus \overline{I_E}^1$ is of course a $1$-finely open set.
Moreover, a positive answer to this open problem would follow from a positive answer
to the previous open problem, just as Proposition \ref{prop:complement of fine closure is quasiopen} follows from Proposition \ref{prop:Choquet type property}.

\begin{corollary}\label{cor:set minus fine boundary is quasiopen}
For any $\mu$-measurable set $E\subset X$, $I_E\setminus \partial^1 I_E$
and $O_E\setminus \partial^1 I_E$ are $1$-quasiopen sets.
\end{corollary}
\begin{proof}
By the definition of the $1$-fine boundary and the second claim of Lemma \ref{lem:inclusion for fine boundaries}, and noting that $\partial^*E=\partial^*I_E$,
\[
X\setminus \overline{I_E}^1=X\setminus (I_E\cup \partial^1 I_E)=O_E\setminus \partial^1 I_E.
\]
Thus by Proposition \ref{prop:complement of fine closure is quasiopen},
$O_E\setminus \partial^1 I_E$ is a $1$-quasiopen set.
Since $\partial^1 I_E=\partial^1 O_E$ by Lemma \ref{lem:coincidence of fine boundaries}, $I_E\setminus \partial^1 I_E$ is also a $1$-quasiopen set.
\end{proof}

\begin{proposition}\label{prop:curves and 1 fine boundary}
Let $E\subset X$ be a $\mu$-measurable set and let $\Gamma$ be the family of curves for which $\gamma(0)\in I_E$ and $\gamma(\ell_{\gamma})\in O_E$, but $\gamma$ does not intersect $\partial^1 I_E$. Then $\Mod_1(\Gamma)=0$.
\end{proposition}

\begin{proof}
By Corollary \ref{cor:set minus fine boundary is quasiopen}, $I_E\setminus \partial^1 I_E$
and $O_E\setminus \partial^1 I_E$ are $1$-quasiopen sets.
Thus by \cite[Remark 3.5]{S2} they are also \emph{$1$-path open} sets,
meaning that for $1$-almost every curve $\gamma$, the sets
$\gamma^{-1}(I_E\setminus \partial^1 I_E)$ and $\gamma^{-1}(O_E\setminus \partial^1 I_E)$
are relatively open subsets of $[0,\ell_{\gamma}]$.
Pick such $\gamma$, and suppose that $\gamma(0)\in I_E$ and $\gamma(\ell_{\gamma})\in O_E$. Since $[0,\ell_{\gamma}]$ is connected,
there exists $t\in (0,\ell_{\gamma})$ for which
\begin{align*}
\gamma(t)
&\in (X\setminus (I_E\setminus \partial^1 I_E)) \cap (X\setminus (O_E\setminus \partial^1 I_E))\\
&= (O_E\cup \partial^*E \cup \partial^1 I_E)\cap (I_E\cup \partial^*E \cup \partial^1 I_E)\\
&= (O_E \cup \partial^1 I_E)\cap (I_E \cup \partial^1 I_E)\\
&= \partial^1 I_E.
\end{align*}
\end{proof}

Using Proposition \ref{prop:curves and 1 fine boundary}, we can now prove Theorem
\ref{thm:sufficiency} by an argument very similar to one used previously in \cite[Theorem 6.5]{LaSh}.

\begin{proof}[Proof of Theorem \ref{thm:sufficiency}]

Fix $\eps>0$. Since $\mathcal{H}(\partial^1 I_E\cap \Omega)<\infty$, we can find a covering of $\partial^1 I_E\cap\Omega$ by balls $B_j=B(x_j,r_j)$ with radii $r_j\le \eps$ 
such that 
 \[
   \sum_{j\in\N}\frac{\mu(B_j)}{r_j}\le \mathcal{H}(\partial^1 I_E\cap\Omega)+\eps.
 \]
Denote $2 B_j:=B(x_j,2 r_j)$.
For each ball $B_j$ in the cover, we fix a $1/r_j$-Lipschitz function
$v_j$ such that $0\le v_j\le 1$ on $X$, $v_j=1$ in $B_j$, and  $v_j=0$ in $X\setminus 2B_j$. Now let
\[
u(x):=
\begin{cases} 
1&\text{ if }x\in I_E,\\
\min\left\{1, \sum_{j\in\N} v_j(x)\right\} &\textrm{ otherwise.}
\end{cases}
 \]
Furthermore, let $v(x):=\min\left\{1,\sum_{j\in\N}v_j(x)\right\}$, and
 \[
  g:=\sum_{j\in\N}\frac{1}{r_j}\ch_{2B_j}.
 \]
Clearly $g$ is an upper gradient of $v$. We will show that $g$ is also
a $1$-weak upper gradient of $u$ in $\Omega$.
Take a curve $\gamma\notin \Gamma$ in $\Omega$ with end points $x,y\in \Om$, where $\Gamma$ was defined in 
Proposition \ref{prop:curves and 1 fine boundary}. If $x,y\in \Omega\setminus I_E$, then 
\[
|u(x)-u(y)|=|v(x)-v(y)|\le \int_\gamma g\, ds.
\]
If $x,y\in I_E\cap \Omega$,
then $u(x)=u(y)$, and hence the upper gradient inequality
\[
|u(x)-u(y)|\le \int_\gamma g\, ds
\]
is satisfied.

Finally, if $x\in I_E\cap \Omega$ and $y\in \Omega\setminus I_E$, then since
$\gamma\notin\Gamma$, there is some $t\in[0,\ell_{\gamma}]$ such that 
$\gamma(t)\in\partial^1 I_E$, and thus $\gamma(t)\in B_k$ for some $k\in\N$. Note that $u(\gamma(0))=u(x)=1$,
$u(\gamma(t))=v(\gamma(t))=1$, and $u(y)=v(y)$.
It follows that
\begin{align*}
|u(x)-u(y)|&\le |u(\gamma(0))-u(\gamma(t))|+|u(\gamma(t))-u(\gamma(\ell_{\gamma}))|\\
&=|v(\gamma(t))-v(\gamma(\ell_{\gamma}))|\le \int_\gamma g\, ds.
\end{align*}
Thus in all cases the pair $u,g$ satisfies the upper gradient inequality for
$\gamma\notin\Gamma$ in $\Omega$, and so $g$ is a $1$-weak upper gradient of $u$
in $\Om$. Furthermore, 
\[
\int_{\Omega} g\, d\mu\le \sum_{j\in\N}\frac{\mu(2B_j)}{r_j}\le C_d\sum_{j\in\N}\frac{\mu(B_j)}{r_j}
\le C_d( \mathcal{H}(\partial^1 I_E\cap \Omega)+\eps)<\infty. 
\]
Now for each $i\in\N$, use the above construction to obtain functions $u_i\in N^{1,1}_{\loc}(\Omega)$, $g_i\in L^1(\Omega)$ corresponding to $\eps=1/i$.
Note that in order to show that $P(E,\Omega)<\infty$, it is enough to show that
$u_i\to \ch_E$ in $L^1(\Omega)$ and that $\Vert g_i\Vert_{L^1(\Omega)}$ is a bounded sequence, since for every $u_i\in N^{1,1}_{\loc}(\Omega)$ we can find a function $w_i\in\liploc(\Omega)$ with $\Vert w_i-u_i\Vert_{N^{1,1}(\Omega)}<1/i$, see \cite[Theorem 5.47]{BB}.
The sequence $\Vert g_i\Vert_{L^1(\Omega)}$ is clearly bounded, and moreover (note that below, the balls $B_j$ also depend on $i$)
\begin{align*}
\int_{\Omega}|u_i-\ch_E|\, d\mu
&\le \int_{\Omega} \ch_{\bigcup_{i\in\N}2B_j}\, d\mu\le \sum_{j\in\N}\mu(2B_j)\le \frac 1i\sum_{j\in\N} \frac{\mu(2B_j)}{r_j}\\
&\le (\mathcal{H}(\partial^1 I_E\cap \Omega)+1)/i\to 0
\end{align*}
as $i\to \infty$.
\end{proof}

\begin{remark}
We still do not know whether $\mathcal H(\partial^*E\cap \Omega)<\infty$ implies
$P(E,\Omega)<\infty$.
It is also not clear whether Proposition \ref{prop:curves and 1 fine boundary} would hold
with $\partial^1 I_E$ replaced by $\partial^*E$. If the answer to the latter question is yes, however, the proof would need to be something different, since Corollary \ref{cor:set minus fine boundary is quasiopen} is not true with $\partial^1 I_E$ replaced by $\partial^*E$, as demonstrated by the latter part of Example \ref{ex:real line}.
\end{remark}

\noindent Address:\\

\noindent Department of Mathematical Sciences\\
4199 French Hall West\\
University of Cincinnati\\
2815 Commons Way\\
Cincinnati, OH 45221-0025\\
E-mail: {\tt lahtipk@ucmail.uc.edu}


\begin{thebibliography}{ACMM}

\bibitem{A1}L.~Ambrosio,
\textit{Fine properties of sets of finite perimeter in doubling metric measure spaces},
Calculus of variations, nonsmooth analysis and related topics.
Set-Valued Anal. 10 (2002), no. 2-3, 111--128.

\bibitem{AFP}L. Ambrosio, N. Fusco, and D. Pallara,
\textit{Functions of bounded variation and free discontinuity problems.}
Oxford Mathematical Monographs. The Clarendon Press, Oxford University Press, New York, 2000.

\bibitem{AMP}L.~Ambrosio, M.~Miranda, Jr., and D.~Pallara,
\textit{Special functions of bounded variation in doubling metric measure spaces},
Calculus of variations: topics from the mathematical heritage of E. De Giorgi, 1--45,
Quad. Mat., 14, Dept. Math., Seconda Univ. Napoli, Caserta, 2004.

\bibitem{AT}L. Ambrosio and P. Tilli,
\textit{Topics on analysis in metric spaces},
Oxford Lecture Series in Mathematics and its Applications, 25. Oxford University Press, Oxford, 2004. viii+133 pp.

\bibitem{BB}A. Bj\"orn and J. Bj\"orn,
\textit{Nonlinear potential theory on metric spaces},
EMS Tracts in Mathematics, 17. European Mathematical Society (EMS), Z\"urich, 2011. xii+403 pp.

\bibitem{BB-OD}A. Bj\"orn and J. Bj\"orn,
\textit{Obstacle and Dirichlet problems on arbitrary nonopen sets in metric spaces, and fine topology},
Rev. Mat. Iberoam. 31 (2015), no. 1, 161--214. 

\bibitem{BB-VC}A. Bj\"orn and J. Bj\"orn,
\textit{The variational capacity with respect to nonopen sets in metric spaces},
Potential Anal. 40 (2014), no. 1, 57--80.

\bibitem{BBL-CCK}A. Bj\"orn, J. Bj\"orn, and V. Latvala,
\textit{The Cartan, Choquet and Kellogg properties
for the fine topology on metric spaces},
to appear in J. Anal. Math.

\bibitem{BBL-WC}A. Bj\"orn, J. Bj\"orn, and V. Latvala,
\textit{The weak Cartan property for the p-fine topology on metric spaces}, 
Indiana Univ. Math. J. 64 (2015), no. 3, 915--941. 

\bibitem{EvaG92}L. C. Evans and R. F. Gariepy,
\textit{Measure theory and fine properties of functions},
Studies in Advanced Mathematics series, CRC Press, Boca Raton, 1992.

\bibitem{Fed}H. Federer,
\textit{Geometric measure theory},
Die Grundlehren der mathematischen Wissenschaften, Band 153 Springer-Verlag New York Inc., New York 1969 xiv+676 pp. 

\bibitem{Giu84}E. Giusti,
\textit{Minimal surfaces and functions of bounded variation},
Monographs in Mathematics, 80. Birkh\"auser Verlag, Basel, 1984. xii+240 pp.

\bibitem{Hj}P. Haj\l{}asz,
\textit{Sobolev spaces on metric-measure spaces},
Heat kernels and analysis on manifolds, graphs, and metric spaces (Paris, 2002), 173--218,
Contemp. Math., 338, Amer. Math. Soc., Providence, RI, 2003.

\bibitem{HaKi}H. Hakkarainen and J. Kinnunen,
\textit{The BV-capacity in metric spaces},
Manuscripta Math. 132 (2010), no. 1-2, 51--73.

\bibitem{HKLS}H. Hakkarainen, R. Korte, P. Lahti, and N. Shanmugalingam,
\textit{Stability and continuity of functions of least gradient},
Anal. Geom. Metr. Spaces 3 (2015), Art. 9.

\bibitem{HKM}J. Heinonen, T. Kilpel\"ainen, and O. Martio,
\textit{Nonlinear potential theory of degenerate elliptic equations},
Unabridged republication of the 1993 original. Dover Publications, Inc., Mineola, NY, 2006. xii+404 pp.

\bibitem{HK}J. Heinonen and P. Koskela,
\textit{Quasiconformal maps in metric spaces with controlled geometry},
Acta Math. 181 (1998), no. 1, 1--61.

\bibitem{HKST}J. Heinonen, P. Koskela, N. Shanmugalingam, and J. Tyson,
\textit{Sobolev spaces on metric measure spaces.
An approach based on upper gradients},
New Mathematical Monographs, 27. Cambridge University Press, Cambridge, 2015. xii+434 pp.

\bibitem{KKST}J. Kinnunen, R. Korte, N. Shanmugalingam, and H. Tuominen,
\textit{A characterization of Newtonian functions with zero boundary values},
Calc. Var. Partial Differential Equations 43 (2012), no. 3-4, 507--528.

\bibitem{KoLa}R. Korte and P. Lahti,
\emph{Relative isoperimetric inequalities and sufficient conditions for finite perimeter on metric spaces},
Ann. Inst. H. Poincar\'e Anal. Non Lin\'eaire 31 (2014), no. 1, 129--154. 

\bibitem{L}P. Lahti,
\textit{A notion of fine continuity for $\BV$ functions on metric spaces},
to appear in Potential Analysis.

\bibitem{L2}P. Lahti,
\textit{Strong approximation of sets of finite perimeter in metric spaces},
preprint 2016, https://arxiv.org/abs/1611.06162

\bibitem{LaSh}P. Lahti and N. Shanmugalingam,
\textit{Fine properties and a notion of quasicontinuity for $\BV$ functions on metric spaces},
to appear in J. Math. Pures Appl.

\bibitem{MZ}J. Mal\'y and W. Ziemer,
\textit{Fine regularity of solutions of elliptic partial differential equations},
Mathematical Surveys and Monographs, 51. American Mathematical Society, Providence, RI, 1997. xiv+291 pp.

\bibitem{M}M.~Miranda, Jr.,
\textit{Functions of bounded variation on ``good'' metric spaces},
J. Math. Pures Appl. (9) 82  (2003),  no. 8, 975--1004.

\bibitem{S2}N. Shanmugalingam,
\textit{Harmonic functions on metric spaces},
Illinois J. Math. 45 (2001), no. 3, 1021--1050.

\bibitem{S}N.~Shanmugalingam,
\textit{Newtonian spaces: An extension of {S}obolev spaces to metric measure spaces},
Rev. Mat. Iberoamericana 16(2) (2000), 243--279.

\bibitem{Zie89}W. P. Ziemer,
\textit{Weakly differentiable functions. Sobolev spaces and functions of bounded variation}, Graduate Texts in Mathematics, 120. Springer-Verlag, New York, 1989. 

\end{thebibliography}
\end{document}